\documentclass[10pt]{amsart}
\usepackage{amsmath,amssymb,enumerate}

\numberwithin{equation}{section}       
 \theoremstyle{plain}    
 \newtheorem{thm}{Theorem}[section]
 \numberwithin{equation}{section} 
 \numberwithin{figure}{section} 
 \theoremstyle{plain}
 \theoremstyle{plain}    
 \theoremstyle{plain}    
 \newtheorem{pro}[thm]{Proposition} 
 \theoremstyle{plain}    
 \theoremstyle{plain}
 \newtheorem{rem}[thm]{Remark}
\theoremstyle{plain} 
\newtheorem{ex}[thm]{Example}

\newtheorem*{thmA}{Theorem A}

\newtheorem*{proB}{Proposition B} 
\newtheorem*{proC}{Proposition C} 
\newtheorem*{pro*}{Proposition} 

\newtheorem*{defi*}{Definition} 
\newtheorem*{ex*}{Example} 

\theoremstyle{plain}
\newtheorem{defi}[thm]{Definition}

\newcommand{\C}{{\mathbb{C}}}

\newcommand{\R}{{\mathbb{R}}}

\newcommand{\MA}{\mathrm{MA}}
\newcommand{\MAH}{\mathrm{MAH}}

\newcommand{\cE}{{\mathcal{E}}}

\renewcommand{\a}{\alpha}
\renewcommand{\b}{\beta}

\newcommand{\f}{\varphi}

\newcommand{\Capis}{{\rm Cap}_{\psi}}
\newcommand{\psh}{{\rm PSH}}

\newcommand{\Amp}{\mathrm{Amp}\,}

\newcommand{\vol}{\operatorname{vol}}

\newcommand{\pistar}{\pi^{\star}}
\newcommand{\pistarb}{\pi_{\star}}
\newcommand{\te}{\theta}
\begin{document}

\setcounter{tocdepth}{2}

\title{Finite Pluricomplex energy measures}


\author{Eleonora Di Nezza}

\address{Department of Mathematics, Imperial College London London, SW7 2AZ, United Kingdom}

\email{e.di-nezza@imperial.ac.uk}

\begin{abstract}
We investigate probability measures with finite pluricomplex energy. We give criteria insuring that a given measure has finite energy and test these on various examples. We show that this notion is a biholomorphic but not a bimeromorphic invariant.
\end{abstract} 

\maketitle

\tableofcontents

\newpage
\section*{Introduction}
Let $X$ be a compact K\"ahler manifold and pick $\alpha\in H^{1,1}(X,\R)$ a big cohomology class. In \cite{bbgz} the authors have defined the \emph{electrostatic energy} $E^*(\mu)$ of a probability measure $\mu$ on $X$ which is a  pluricomplex analogue of the classical logarithmic energy of a measure.\\
\indent They then give a useful caracterization (that for our purposes we will take as definition) of measures $\mu$ with finite energy:
\begin{defi*}
A non-pluripolar probability measure $\mu$ has finite energy in a big class $\alpha$ if and only if there exists $T\in \cE^1(X,\alpha)$ such that 
$$\mu=\frac{\langle T^n \rangle}{\vol(\alpha)}.$$
In this case we write $\mu\in \MA(\cE^1(X, \alpha))$.
\end{defi*}
We recall that $\langle T^n \rangle$ is the so called non-pluripolar measure and the class $\cE(X,\alpha)$ is the set of positive closed currents $T\in \alpha$ with \emph{full Monge-Amp\`ere mass}, i.e. such that $\int_X \langle T^n\rangle=\vol(\alpha)$, while $\cE^1(X,\alpha)\subset \cE(X,\alpha)$ is the set of currents having finite energy. More precisely, we say that $T\in \cE(X,\alpha)$ has finite energy iff
$$\int_X |\f-V_{\theta}|\langle T^n\rangle<+\infty$$
where $T=\theta+dd^c\f$ and $\theta+dd^c V_\theta$ is a current with minimal singularities in $\alpha$.\\

Measures with finite energy have played a crucial role when solving degenerate complex Monge-Amp\`ere equations using the variational approach (see \cite{bbgz}). In this note we investigate such a property and we give some concrete examples/counterexamples of measures having finite energy.\\
\indent In particular, we wonder whether this notion is a bimeromorphic invariant. It turns out that it is invariant under biholomorphisms but not under bimeromorphisms. The latter is a very subtle point and it will be explained in details.

We also consider the dependence on the cohomology classes. We prove the following:
\begin{thmA} 
The finite energy condition is a biholomorphic invariant but it is not, in general, a bimeromorphic invariant.
\end{thmA} 
We prove more generally that for any $\alpha,\beta$ K\"ahler classes
$$\mu\in \MA(\cE^1(X,\alpha)) \Longleftrightarrow \mu\in \MA(\cE^1(X,\beta)).$$

\begin{proB}
Let $\pi: X\rightarrow \mathbb{P}^2$ be the blow-up at one point of the complex projective plane. Then there exists a probability measure $\mu$ and a K\"ahler class $\{\tilde{\omega}\}$ on $X$ such that 

$$\mu\in \MA\left(\cE^1(X,\{\tilde{\omega}\})\right) \quad \mbox{but} \quad\mu\notin \MA\left(\cE^1(X,\{\pistar \omega_{FS}\})\right)$$
and furthermore, $ \pistarb\mu\notin \MA\left(\cE^1(\mathbb{P}^2,\{\pistarb \tilde{\omega} \})\right).$
\end{proB}

\medskip

\indent Working in the  K\"ahler setting and following ideas developed in \cite{DL1, DL2}, we are able to insure that a given non-pluripolar probability measure $\mu$ has finite energy when it is dominated by the generalized Monge-Amp\`ere capacity.
More precisely, if we assume that there exists a constant $A>0$ such that $$\mu \leq A\, \Capis ^{1+\varepsilon}$$
for some $\varepsilon>0$, where  $\psi\in \cE^1(X,\omega/2)$, then $\mu$ has finite energy in $\{\omega\}$ (Proposition \ref{cap}).


\medskip

We also give various criteria insuring that a given probability measure has finite energy (see Section \ref{smoothWeights}, Propositions \ref{toric} and \ref{divisorial}) and consider different types of singular behavior (radial and toric measures, divisorial singularities).\\
\indent Let $D=\sum_j D_j$ is a simple normal crossing divisor and for each $j$, let $s_j$ be the defining section of $D_j$. Fix a hermitian metric $h_j$ on the holomorphic line bundle defined by $D_j$ such  that $|s_j|:=\vert s_j\vert_{h_j} \leq 1/e$. \\
We consider, for example, measures $\mu=f dV$ with densities that can be written as
$$
f= \frac{h}{\prod_{j=1}^n \vert s_j\vert^2 (-\log \vert s_j\vert)^{1+\alpha}}
$$
where $h$ is a bounded function, $1/B\leq h\leq B$ for some $B>0$ and $\alpha>0$. 

\noindent In these special cases we can give a complete caracterization:
\begin{proC}
Let $\omega$ be a K\"ahler form. The following holds:
$$\mu\in \MA(\cE^1(X,\omega)) \quad \mbox{if and only if} \quad \alpha >1/2.$$
\end{proC}

\medskip

Let us describe the contents of the paper. We first recall some definitions and known facts. In Section \ref{FiniteE} we give some concrete examples of measures with finite energy. We then discuss the invariance properties of finite energy measures and we give a tricky counterexample insuring the non invariance under bimeromorhic maps (Section \ref{NonS}).  

\section{Preliminaries}\label{pre}
\subsection{Big classes and the non-pluripolar product}
Let $X$ be a compact K{\"a}hler ma\-ni\-fold of complex dimension $n$ and let $\a\in H^{1,1}(X,\R)$ be a real $(1,1)$-cohomology class.
Recall that $\a$ is said to be \emph{pseudo-effective} (\emph{psef} for short) if it can be represented by a closed positive $(1,1)$-current $T$.
Given a smooth representative $\te$ of the class $\a$, it follows from $\partial\bar{\partial}$-lemma that any positive $(1,1)$-current can be written as $T=\te+dd^c \f$ where the global potential $\f$ is a $\te$-plurisubharmonic ($\te$-psh for short) function, i.e. $\te+dd^c\f\geq 0$. Here, $d$ and $d^c$ are real differential operators defined as
$$d:=\partial +\bar{\partial},\qquad d^c:=\frac{i}{2\pi}\left(\bar{\partial}-\partial \right).$$
The set of all psef classes forms a closed convex cone and its interior is by definition the set of all \emph{big} cohomology classes.\\
\indent We say that the cohomology class $\a$ is \emph{big} if it can be represented by a \emph{K{\"a}hler current}, i.e. if there exists a positive closed $(1,1)$-current $T_+\in\a $ that dominates some (small) K{\"a}hler form. By Demailly's regularization theorem \cite{Dem92} one can assume that $T_+:=\te+dd^c \f_+$ has \emph{analytic singularities}, namely there exists $c>0$ such that (locally on $X$),
$$
\f_+=\frac{c}{2}\log\sum_{j=1}^{N}|f_j|^2+u,
$$
where $u$ is smooth and $f_1,...f_N$ are local holomorphic functions.

\begin{defi}
{\it If $\a$ is a big class, we define its \emph{ample locus} $\Amp(\a)$ as the set of points $x\in X$ such that there exists a strictly positive current $T\in\a$ with analytic singularities and smooth around $x$.}
\end{defi}

Note that the ample locus $\Amp(\a)$ is a Zariski open subset by definition, and it is nonempty since $T_+$ is smooth on a Zariski open subset of $X$.

If $T$ and $T'$ are two closed positive currents on $X$, then $T$ is said to be \emph{more singular} than $T'$ if their local potentials satisfy $\f\le\f'+O(1)$.\\
\indent A positive current $T$ is said to have \emph{minimal singularities} (inside its cohomology class  $\a$) if it is less singular than any other positive current in $\a$. Its $\theta$-psh potentials $\f$ will correspondingly be said to have minimal singularities.\\
\indent Note that any $\theta$-psh function $\f$ with minimal singularities is locally bounded on the ample locus $\Amp(\a)$ since it has to satisfy $\f_+\le \f + O(1)$.
Furthermore, such $\theta$-psh functions with minimal singularities always exist, one can consider for example
$$V_\theta:=\sup\left\{ \f\,\,\theta\text{-psh}, \f\le 0\text{ on } X \right \}.$$

We now introduce the \emph{volume} of the cohomology class $\a\in H_{big}^{1,1}(X,\R)$:
\setcounter{tocdepth}{1}
\begin{defi}
\it{Let $T_{\min}$ a current with minimal singularities in $\a$  and let $\Omega$ a Zariski open set on which the potentials of $T_{\min}$ are locally bounded, then 
\begin{equation}\label{v}
\vol(\a):=\int_{\Omega} T_{\min}^n>0
\end{equation}
is called the volume of $\a$.}
\end{defi}
\indent Note that the Monge-Amp{\`e}re measure of $T_{\min}$
is well defined in $\Omega$ by \cite{bt} and that the volume is independent of the choice of $T_{\min}$ and $\Omega$ (\cite[Theorem 1.16]{begz}).
 
Given $T_1,...,T_p$ closed positive $(1,1)$-currents, it has been shown in \cite{begz} that the (multilinear) \emph{non-pluripolar product} 
$$\langle T_1\wedge...\wedge T_p\rangle$$
is a well defined closed positive $(p,p)$-current that does not charge pluripolar sets. In particular, given $\f_1,...,\f_n$ $\te$-psh functions, we define their Monge-Amp\`ere measure as
$$\MA(\f_1,...\f_n):= \langle (\theta+dd^c\f_1)\wedge...\wedge (\theta+dd^c\f_n)\rangle.$$
By construction the latter is a non-pluripolar measure and satisfies $$\int_X MA(\f_1,...\f_n)\leq \vol(\{\te\}).$$
In the case $\f_1=...=\f_n=\f$ we simply set $$\MA(\f)=MA(\f,...\f).$$ 
By definition of the volume of $\{\theta\}$ and the fact that the non-pluripolar product does not charge pluripolar sets, it is then clear that for any $T_{\min}=\theta+dd^c \f_{\min}\in \{\te\}$ current with minimal singularities, one has $$\int_X \MA(\f_{\min})=\int_X\langle T_{\min}^n\rangle=\vol(\{\te\}).$$

\subsection{Finite (weighted) energy classes}
Let $\a\in H^{1,1}(X,\R)$ be a big class and $\theta\in \a$ be a smooth representative.

\begin{defi}\label{defi:fullMA} 
{\it A closed positive $(1,1)$-current $T$ on $X$ with cohomology class $\a$ is said to have \emph{full Monge-Amp{\`e}re mass} if
$$\int_X\langle T^n\rangle=\vol(\a).$$
We denote by $\cE(X,\a)$ the set of such currents.} 
{\it If $\f$ is a $\theta$-psh function such that $T=\theta+dd^c\f$, we will say that $\f$ has \emph{full Monge-Amp{\`e}re mass} if $\theta+dd^c\f$ has full Monge-Amp{\`e}re mass. We denote by $\cE(X,\theta)$ the set of corresponding functions.}
\end{defi}

Currents with full Monge-Amp{\`e}re mass have mild singularities, in particular they have zero Lelong number at every point $x\in \Amp(\a)$ (see \cite[Proposition 1.9]{dn}).\\

\begin{defi} 
{\it 
We define the energy of a $\theta$-psh function $\f$ as 
$$
E_{\theta}(\f):=\frac{1}{n+1}\sum_{j=0}^n\int_X-(\f-V_{\theta})\langle T^j\wedge \theta_{\min}^{n-j}\rangle \;\in\;]-\infty,+\infty]
$$
with $T=\theta+dd^c\f$ and $\theta_{\min}=\theta+dd^cV_{\theta}$. We set
$$\cE^1(X,\theta):= \{\f\in \cE(X,\theta) \;\,  | \;\, E_{\theta}(\f)<+\infty \}.$$
We denote by $\cE^1(X,\a)$ the set of positive currents in the class $\a$ whose global potentials have finite energy. }
\end{defi} 

The energy functional is non-increasing and for an arbitrary $\theta$-psh function $\f$, 
 $$
 E_{\theta}(\f):=\sup_{\psi\ge\f} E_{\theta}(\psi)\in]-\infty,+\infty]
 $$
over all $\psi\ge\f$ with minimal singularities (see \cite[Proposition 2.8]{begz}). 

\section{Examples of Finite Energy Measures}\label{FiniteE}

Let $X$ be a compact K{\"a}hler manifold of complex dimension $n$ and $\alpha$ be a big class and $\te\in \a$ be a smooth representative. The following notion has been introduced in \cite{bbgz}:

\begin{defi}\label{fe}
A probability measure $\mu$ on $X$ has finite energy in $\a$ iff there exists $T\in \cE^1(X,\alpha)$ such that
\begin{equation}\label{eq1}
\mu= \frac{\langle T^n \rangle}{\vol(\alpha)}.
\end{equation}
 In this case we write $\mu\in \MA(\cE^1(X,\a))$.
\end{defi}
The purpose of this note is to study the set $\MA(\cE^1(X,\a))$ of finite energy measures.
\subsection{Some Criteria}
Let us recall that a probability measure $\mu$ having finite energy is necessarily non-pluripolar (see \cite[Lemma 4.4]{bbgz}). \\
\noindent When $(X,\omega)$ is a compact Riemann surface ( $n=1$) then $\mu=\omega+dd^c\f \in \MA(\cE^1(X,\{\omega\}))$ iff $\f$ belongs to the Sobolev space $W^{1,2}(X)$. This follows from Stokes theorem since 
$$\int_X (-\f)d\mu=\int_X (-\f)\omega+\int_X d\f\wedge d^c \f.$$


We recall that a probability measure $\mu$ has finite energy iff for any $\psi\in \cE^1(X,\te)$
$$\int_X -(\psi-V_\te) d\mu< +\infty,$$
where $V_\te$ is the $\te$-psh function with minimal singularities defined in Section \ref{pre} (see \cite[Lemma 4.4]{bbgz}). In particular, this insures that the set of measures with finite energy in a given cohomology class is convex, since given $\mu_1,\mu_2\in \MA(\cE^1(X,\{\te\}))$, then clearly for any $t\in[0,1]$,
$$\int_X-(\psi-V_\te) \,\left(td\mu_1+(1-t)d\mu_2\right)<+\infty.$$

\indent Let $\mu,\nu$ be two probability measures such that $\mu\lesssim \nu$. An immediate consequence of the above characterization is that $\mu$ has finite energy in $\a$ if so does $\nu$.\\

A technical criterion to insure that a given probability measure has finite energy is the following:
\begin{pro}\label{cap}
Let $\omega\in \a$ be a K{\"a}hler form and $\psi\in \cE^1(X,\omega/2)$. Assume there exists a constant $A>0$ such that $$\mu \leq A\, \Capis ^{1+\varepsilon}$$
for some $\varepsilon>0$. Then $\mu$ has finite energy in $\a$.
\end{pro}
\noindent Here, by $\Capis$ we mean the generalized Monge-Amp\`ere capacity introduced and studied in \cite{DL1,DL2}, namely for any Borel set $E\subset X$,
$$\Capis(E):=\sup \left\{\int_E \MA(u)\,\quad |\quad  u\in\psh(X,\omega),\,\; \psi-1\leq u \leq \psi \right\}.$$

\begin{proof}
Such result follows directly from the arguments in \cite[Theorem 3.1]{DL1}.
Recall that by \cite{gz} there exists a unique (up to constant) $\f\in \cE(X,\omega)$ such that $$\mu=  (\omega+dd^c \f)^n.$$ Set
$$
H(t)=\left[\Capis(\{\varphi<\psi-t\})\right]^{1/n},\ t>0.
$$
Using \cite[Proposition 2.8]{DL1} and the assumption on the measure $\MA(\f)$, we get
$$ 
s H(t+s)\leq A^{1/n} H(t)^{1+\varepsilon},  \ \forall t>0, \forall s\in [0,1]. 
$$
Then by \cite[Lemma 2.4]{EGZ09} we get $\varphi \geq \psi-C$,
where $C$ only depends on $A$. Hence $\f\in \cE^1(X,\omega)$ since the class $\cE^1(X,\omega)$ is stable under the max operation \cite[Corollary 2.7]{gz}.
\end{proof}

\subsection{Measures with densities}  

Let $\a$ be a K{\"a}hler class and $\omega\in \a$ be a K{\"a}hler form. We consider probability measures of the type $\mu=f\omega^n$ with density $0<f\in L^1(X)$. We investigate under which assumptions on the density $f$, the measure $\mu$ has finite energy. We recall that by \cite{gz} there exists a unique (up to constant) $\omega$-psh function $\f\in \cE(X,\omega)$ solving
\begin{equation}\label{eq}
(\omega+dd^c \f)^n=f\omega^n.
\end{equation} 
When $f \in L^p(X)$ for some $p>1$, it follows from the work of Ko{\l}odziej \cite{Kol98}
that the solution of (\ref{eq}) is actually uniformly bounded (and even H\"older continuous) on the whole of $X$. In particular, $\f\in \cE^1(X,\omega)$ that means $\mu\in \MA(\cE^1(X,\{\omega\}))$. In the following we consider concrete cases when the density $f$ is merely in $L^1(X)$.\\

If the density has \emph{finite entropy}, i.e. $\int_X f\log f < +\infty$, then the measure has finite energy (see \cite[Lemma 2.18]{bbgz12}). As we will see in the sequel this condition is still too strong and it is not necessary.

\subsection{Radial measures}\label{smoothWeights}
We consider here radially invariant measures. For simplicity we work in the local case but the same type of computations can be done in the compact setting. Let $\chi:\R\rightarrow \R$ a convex increasing function such that $\chi'(-\infty)=0$ and $\chi(t)=t$ for $t>0$. Denote by $\|z\|=\sqrt{|z_1|^2+...+|z_2|^2}$ the Euclidean norm of $\C^n$. Consider $$\f(z)=\chi\circ \log\|z\|.$$ Then $\f$ is plurisubharmonic in $\mathbb{B}(0, r)\subset \C^n$ with $r>0$ small, and  $$\mu:=(dd^c \f)^n.$$
Observe that, giving a radial measure in $\mathbb{B}(0, r)$ is the same thing as giving a positive measure $\nu$ in the interval $(0,r]$. This means that $\mu$ has finite energy if and only if
$$\int_0^r |\chi(\log \rho)| d\nu(\rho)<\infty.$$

\subsubsection*{Smooth weights} Let $\chi:\R\rightarrow \R$ a smooth convex increasing function such that $\chi'(-\infty)=0$. Denote by $\|z\|=\sqrt{|z_1|^2+...+|z_2|^2}$ the Euclidean norm of $\C^n$. Consider $$\f(z)=\chi\circ \log\|z\|.$$ Then $\f$ is plurisubharmonic in $\mathbb{B}(0, r)\subset \C^n$ with $r>0$ small, and 
$$\mu:=(dd^c \f)^n= f dV, \quad{\rm{with}}\quad f(z)=\frac{c_n (\chi' \circ \log \|z \|)^{n-1} \chi''(\log\|z \|) }{\|z \|^{2n}}   $$
where $dV$ denotes the Euclidean measure on $\C^n$. It turns out that $\mu$ has finite energy iff 
$$\int_{\mathbb{B}(0, r)} -\chi\circ \log\|z\| f(z) dV <+\infty,$$
that, using polar coordinates, is equivalent to 
\begin{equation}\label{cond}
\int_{-\infty}^{\log r} -\chi(s) (\chi'(s))^{n-1} \chi''(s) ds < +\infty. 
\end{equation}
\begin{ex} Consider $\chi_p(t)=-(-t)^p$ with $0<p<1$. Then the associated radial measure $\mu_p$ has finite energy iff $p<\frac{n}{n+1}$. 
\end{ex}

In \cite[Corollary 4.4]{ddghz}, the authors have proven that the range of $\MAH(X,\omega)$, the Monge Amp\`ere operator of plurisubsharmonic H\"older continuous functions, has the $L^p$ property:  if $\mu\in \MAH(X,\omega) $ and $0 \leq g \in L^p(\mu)$ for some $ p > 1$ with $\int_X g d\mu = \int_X\omega^n $, then $g\mu \in \MAH(X, \omega)$. One can wonder whether $\MA(\cE^1(X,\omega))$, i.e. the set of finite energy measures, satisfies such a property. This is not the case as the following example shows.
\begin{ex} Let $n>1$ and $\mu=f\omega^n=(\omega+dd^c\chi\circ \log \|z\|)^n$ where $\chi(t):=-(-t)^{\frac{n-1}{n+1}}$. Then $\mu\in \MA(\cE^1(X,\omega))$. We now consider $g(z)= (-\log\|z\|)^{n/(n+1)}$ and observe that $g\in L^{\frac{n+1}{n}}{(\mu)}$. But then $g \mu \notin \MA(\cE^1(X,\omega))$ since one can check that $$g\mu \sim (\omega+dd^c \chi_1 \circ \log\|z\|)^n,$$ where $\chi_1(t)=-(-t)^{n/(n+1)}$ and then the integral in (\ref{cond}) is not finite.
\end{ex}

\subsection{Toric measures}
Let $T$ be the torus $(\mathbb{S}^1)^n$ in $\C^n$ of real dimension $n$ and denote by $T_c=(\C^*)^n$ its compactification. Recall that a $n$-dimensional compact K\"ahler toric manifold $(X,\omega, T_c)$ is an equivariant compactification of the torus $T_c$ equipped with a $T$-invariant K\"ahler metric $\omega$ which writes
$$\omega=dd^c \psi \quad \rm{in} \; (\C^*)^n$$
where $\psi$ is a psh $T$-invariant function, hence $\psi(z)= \phi_P \circ L(z)$ with
$$L:z\in (\C^*)^n\rightarrow (\log|z_1|,\cdots, \log|z_n|)\in \R^n$$
and $\phi_P: \R^n \rightarrow \R$ a strictly convex function.\\
\indent In the same way, given a $\omega$-psh \emph{toric} potential $\varphi$ on $X$ (i.e. a $\omega$-psh function that is $T$-invariant), we can consider the corresponding convex function $\phi$ in $\R^n$ such that
$$\phi\circ L= \phi_P\circ L+ \varphi \quad \rm{in}\; (\C^*)^n.$$

Recall that a famous result of Atiyah-Guillemin-Sternberg claims that the moment map $\nabla \phi_P: \R^n\rightarrow \R^n$ sends $\R^n$ to a bounded convex polytope $P$, which is independent of $\phi_P$.\\

\indent The \emph{Legendre transform} of a convex function $\phi(x)$ is defined by
$$\phi^*(p):= \sup_{x\in \R^n} \{\langle x,p\rangle -\phi(x)\}$$
which is a convex function in $\R^n$ with values in $(-\infty, +\infty ].$

Now, let $\phi_P^*$ denote the Legendre transform of $\phi_P$. Observe that $\phi_P^*=+\infty$ in $\R^n\setminus P$ and for $p\in \rm{int}(P)$, 
$$\phi^*_P(p)=\langle x, p\rangle-\phi_P(x)\quad  {\rm{with}} \quad \nabla \phi_P (x)=p  \Leftrightarrow \nabla \phi^*_P(p)=x .$$
In \cite{g}, the author shows that, given a $\omega$-psh toric potential $\f$, we can read off the singular behavior of $\varphi$ from the integrability properties of the Legendre transform of its associates convex function. More precisely, he proves that
$$\f\in \cE^q_{toric}(X, \omega) \Longleftrightarrow \phi^* \in L^q(P, dp)$$
for any $q>0$ (see also \cite{bb13} for a proof in the case $q=1$), where $\cE^q_{toric}(X, \omega)$ is the set of $T$-invariant $\omega$-psh functions that belong to the energy class $\cE^q(X, \omega)$.\\

Note that to any non-pluripolar $T$-invariant measure $\mu$ on $X$ of total mass $\int_X\omega^n$, we can associate a measure $\tilde{\mu}$ on $\R^n$ of total mass $\vol(P)$. Indeed, by \cite{gz} there exists a unique (up to constant) $\f\in \cE_{toric}(X,\omega)$ such that $\mu=\MA(\f)$ and then $\tilde{\mu}=\MA_{\R} (\phi)$ where $\MA_{\R}$ denotes the real Monge-Amp\`ere measure of the convex function $\phi$ associated to $\f$.\\

In the result below we establish some regularity of the potential $\f$ in terms of a moment condition on $\tilde{\mu}$.
\begin{pro}\label{toric}
Let $n>1$. Assume $\int_{\R^n} |x|^q d\tilde{\mu} <+\infty $ for some $1\leq q< n.$ Then $\f\in \cE^{q^*}_{toric}(X,\omega)$ with $q^*=\frac{qn}{n-q}$. In particular, this implies that $\mu$ has finite energy.
\end{pro}
\begin{proof}
Let $\phi_j$ a sequence of strictly convex smooth fucntions decreasing to $\phi.$
Making the change of variables $p=\nabla \phi_j (x)$, which by duality means that $x=\nabla\phi_j^*(p)$, gives $$\int_{\R^n} |x|^q \MA_{\R}(\phi_j)=\int_P |\nabla \phi_j^*|^q dp.$$
By Sobolev inequality there exists a uniform constant $C>0$ such that for any $j$,
$$\|\phi_j^*\|_{L^{q^*}(P)}\leq C \|\nabla \phi_j^*\|_{L^q(P)}.$$
Passing to the limit for $j\rightarrow +\infty$, we obtain $\phi^*\in L^{q^*}(P)$. The last statement simply follows from the fact that $q^*>q\geq 1.$
\end{proof}
We also refer the reader to \cite[Theorem 2.19]{bb13} for an alternative proof of the fact that when $\tilde{\mu}$ has finite first moment, i.e. $\int_{\R^n} |x| d\tilde{\mu}<+\infty$, then the measure has finite energy.\\
Asking the measure $\tilde{\mu}$ to have finite first moment is a sufficient condition for the measure to have finite energy but it is not necessary as the following example shows:
\begin{ex}
Assume $X=\C\mathbb{P}^1$ is the Riemann sphere and $\omega$ is the Fubini-Study K\"ahler form,
$$\phi_P(x)=\frac{1}{2} \log [1+e^{2x}]$$
and $P=[0,1]$. For any $\beta\in (0,1)$, consider the convex function
$$\phi_\beta(x)=\begin{cases} x-1 &\mbox{if} \;x>0 \\
                         1  &\mbox{if} \;x=0 \\
                         -C(-x)^\beta  &\mbox{if}\; x<0
\end{cases}
 $$
 Then $\phi^*_\beta (p)= p^{-\beta^*},$ where $\beta^*=\beta/(1-\beta)>0$. In this case the toric measure $\mu=\MA(\f_\beta)$ on $X$ associated to the measure $\MA_{\R}(\phi_\beta)$ on $\R^n$, has finite energy as soon as $\beta^*<1$, or equivalently $\beta<1/2$, although the first moment condition is never satisfied. In this example, the measure has finite energy if and only if the moment condition of order $1/2$ is finite.\\
 \begin{rem}
We expect, more generally, that the moment condition of order $n/(n+1)$ is a necessary and sufficient condition to insure that a given toric measure has finite energy.
\end{rem}
\end{ex} 
\subsection{Divisorial singularities}

Let $D= \sum_{j=1}^{N} D_j$ be a simple normal crossing divisor on $X$. Here "simple normal crossing" means that around each intersection point of $k$ components $D_{j_1},...,D_{j_k}$ ($k\leq N$), we can find complex coordinates $z_1,...,z_n$ such that for each $l=1,...,k$ the hypersurface $D_{j_l}$ is locally given by $z_l=0$. For each $j$, let $L_j$ be the holomorphic line bundle defined by $D_j$. Let $s_j$ be a holomorphic section of $L_j$ defining $D_j$, i.e 
$D_j=\{s_j=0\}$. We fix  a hermitian metric $h_j$ on $L_j$ such that 
$\vert s_j\vert:= \vert s_j\vert_{h_j}\leq 1/e$. 

We say that $f$ satisfies Condition $C(B,\alpha)$ for some $B>0$, $\alpha>0$ if
\begin{equation}\label{eq: cond f log bis}
f= \frac{h}{\prod_{j=1}^N \vert s_j\vert^2 (-\log \vert s_j\vert)^{1+\alpha}}.
\end{equation}
where $h\in C^{\infty}(X)$, $1/B\leq h\leq B$.
\begin{pro}\label{divisorial}
Assume that $f$ satisfies $C(B,\alpha)$ for some $B>0$, $\a>0$.Then
$$\mu\in \MA(\cE^1(X,\{\omega\})) \quad \mbox{if and only if}\quad \alpha>1/2. $$
\end{pro}
\begin{proof}
When $\a>1/2$, by \cite[Theorem 2]{DL1} we can find $q\in(1-\a, 1/2)$ such that 
$$\sum_{j=1}^N -a_1(-\log \vert s_j\vert)^q-A_1\leq \f,$$
where $a_1, A_1>0$ depends on $B,\alpha, q$. Note that the function $u_p=\sum_{j=1}^N -a_1(-\log \vert s_j\vert)^q$ if $\omega$-psh is $a_1>0$ is small enough and that $u_q\in \cE^1(X,\{\omega\})$, hence so does $\f$.\\
In the case $\a\in (0,1)$, by \cite[Proposition 4.4]{DL1} we get that for each $0<p<1-\a$ we have
$$
 \varphi \leq \sum_{j=1}^N -a_2(-\log \vert s_j\vert)^p +A_2,
$$
where $a_2, A_2>0$ depend on $B,\alpha,p$. Denote $u_p=\sum_{j=1}^N -a_2(-\log \vert s_j\vert)^p$. Observe that if $\a<1/2$, we can choose $p\in (1/2, 1-\a)$ such that $u_p\notin \cE^1(X, \omega)$. Thus $\f\notin \cE^1(X, \omega)$ and hence the conclusion. What is missing is the case $\alpha=1/2$. Consider $u=\sum_{j=1}^N -b(-\log \vert s_j\vert)^{1/2}$, where $b$ is a small constant such that $u\in \psh(X,\omega)$. Then $u\notin \cE^1(X, \omega)$ and we can find a constant $C>0$ such that 
$$\MA(u)\leq \frac{C}{B\prod_{j=1}^N \vert s_j\vert^2 (-\log \vert s_j\vert)^{3/2}},$$
 hence the conclusion.
\end{proof}
\begin{rem}
{\it Observe that in this case the \emph{entropy condition}, $\int_X f\log f<+\infty$, is satisfied only for $\a>1$ although the measure has finite measure as soon as $\a>1/2$.}
\end{rem}

\section{Stability Properties}\label{NonS}
Given $X, Y$ compact K{\"a}hler manifolds of complex dimension $n,m$, respectively,  with $m\leq n$ and $f: X\rightarrow Y$ a holomorphic map, it is the finite energy property is preserved under $f$. \\
It turns out that finite energy measures are invariant under biholomorphisms but not under bimeromorphisms as we explain in Sections \ref{I} and \ref{nonI}.\\
\indent In the following we wonder whether this notion depends or not on the cohomology class. In other words, given $\alpha,\beta$ big classes and a probability measure $\mu\in \MA(\cE^1(X,\a))$, we ask whether $\mu\in \MA(\cE^1(X,\b))$ or not. \\
We recall that by \cite[Theorem 3.1]{begz}, there exists a unique positive current $S\in \cE(X,\b)$ such that $$\mu=\frac{\langle S^n \rangle}{\vol(\beta)}.$$ Therefore the question reduces to asking whether $S\in \cE^1(X,\beta)$. It turns out that this is false in general (see Example \ref{nofin}). We obtain a positive answer under restrictive conditions on the cohomology classes, i.e. $\a,\b$ both  K{\"a}hler, as Proposition \ref{prok} shows.
\subsection{Invariance property}\label{I}
Finite energy measures are invariant under biholomorphisms. Let $f: X\rightarrow Y$ be a biholomorphic map (in particular $n=m$)

\begin{pro}\label{prok}
Let $\a,\b$ be K{\"a}hler classes and $\mu$ a probability measure. Then
$$\mu\in \MA(\cE^1(X,\a)) \Longleftrightarrow \mu\in \MA(\cE^1(X,\b)).$$
\end{pro}
\begin{proof}
Pick $\omega_1$ and $\omega_2$ K{\"a}hler forms in $\a$ and $\b$, respectively. We suppose $\mu\in \MA(\cE^1(X,\a))$ and we write $$\mu=\frac{(\omega_1+dd^c\f_{\mu})^n}{\vol(\a)}.$$ We want to show that there exists $\psi_\mu\in \cE^1(X,\omega_2)$ such that $\mu=\frac{(\omega_1+dd^c\psi_{\mu})^n}{\vol(\b)}$. By \cite[Theorem 4.2]{gz}, it is equivalent to showing that $\cE^1(X,\omega_2)\subset L^1(\mu)$. We recall that since $\omega_1,\omega_2$ are K{\"a}hler forms, there exists $C>1$ such that $\omega_1\leq C\omega_2$. Now, for all $\psi\in \cE^1(X,\omega_2)$, $\psi\leq 0$, 
$$\int_X (-\psi) d\mu =\frac{1}{\vol(\a)}\int_X (-\psi) (\omega_1+dd^c\f_{\mu})^n \leq \frac{1}{\vol(\a)}\int_X (-\psi) (C\omega_2+dd^c\f_{\mu})^n<+\infty.$$
The finiteness of the above integral follows from \cite[Proposition 2.5]{gz} and from the fact that \cite[Theorem 3.1]{dn} insures $\psi,\f_\mu\in \cE^1(X,C\omega_2)$.
\end{proof}

\subsection{Non invariance property}\label{nonI}

The notion of finite energy for non pluripolar measures is not invariant under bimeromorphic change of coordinates. Indeed, the Proposition below points out that $\mu\in \MA(\cE^1(X, \{\tilde{\omega}\}))$ but $\pi_\star \mu\notin \MA(\cE^1(\mathbb{P}^2,\{ \omega_{FS}\}))$.  \\
\indent More generally, Definition \ref{fe} depends on the cohomology class: in the following we show that there exist a measure $\mu$ and big classes $\a, \b\in H^{1,1}(X,\R)$ such that $\mu\in \MA(\cE^1(X,\a))$ but $\mu\notin \MA(\cE^1(X,\b))$.

\begin{pro}\label{nofin}
Let $\pi: X\rightarrow \mathbb{P}^2$ be the blow-up at one point $p$ and set $E:=\pi^{-1}(p)$. Then there exists a probability measure $\mu$ and a K\"ahler class $\{\tilde{\omega}\}$ on $X$ such that 

$$\mu\in \MA\left(\cE^1(X,\{\tilde{\omega}\})\right) \quad \mbox{but} \quad\mu\notin \MA\left(\cE^1(X,\{\pistar \omega_{FS}\})\right)$$
and furthermore, $ \pistarb\mu\notin \MA\left(\cE^1(\mathbb{P}^2,\{\pistarb \tilde{\omega} \})\right).$
\end{pro}

\begin{proof}
Let $U$ be a local chart of $\mathbb{P}^2$ such that $p \rightarrow (0, 0)\in U$. Fix a positive $(1,1)$-current $\omega'$ on $ \mathbb{P}^2$ such that its global potential on $U$ can be written as $\varepsilon \chi(z) \log\|z\|$ where $\chi$ is a cut-off fucntion so that $\chi\equiv 1$ near $(0,0)$ and $\varepsilon>0$. Then $\tilde{\omega}:=(\pistar \omega'-[E])+\pistar \omega_{FS}$ is a K{\"a}hler form and clearly $\tilde{\omega}\geq \pistar \omega_{FS}$. Let $\a=\{\tilde{\omega}\}$ and $\beta=\pistar \{\omega_{FS}\}$ with $\vol(\omega_{FS})=1$. On $U$ we define 
$$\f_{p} := \frac{1}{C}\chi \cdot u_{p} -K_{p}$$
where $u_{p} := -(- \log\|z\|)^{p}$, $\chi$ is a smooth cut-off function such that $\chi \equiv 1$ on $\mathbb{B}$ and $\chi \equiv 0$ on $U \setminus \mathbb{B}(2)$, $K_{p}$ is a positive constant such that $\f_{p}\leq −1$ and $C>0$. Choosing $C$ big enough $\f_{p}$ induces a $\omega_{FS}$-psh function on $\mathbb{P}^2$, say $\tilde{\f}_{p}$.
For $p=\frac{1}{2}-\delta$ with $\delta>0$ small enough, we set
$$\mu:=\frac{(\tilde{\omega}+dd^c \pistar \tilde{\f}_p)^2}{\vol(\tilde{\omega})}.$$ We will show that $\mu\notin \MA(\cE^1(X,\b))$, or better that there exists a function $\psi\in \cE^1(X,\pistar \omega_{FS})$, $\psi\leq 0$, such that $\int_X (-\psi)d\mu = +\infty$ (see \cite[Theorem 4.2]{gz}). We pick $\psi:=\pistar \tilde{\f}_\varepsilon$ with $\varepsilon=\frac{2}{3}-\delta'$, $\delta'>0$ small enough. Observe that $\psi\in \cE^1(X,\pistar \omega_{FS})$ but $\psi\notin \cE^1(X,\tilde{\omega})$ (see \cite[Example 3.5]{dn}). We claim that $\int_X (-\pistar \tilde{\f}_\varepsilon) (\tilde{\omega}+dd^c \pistar \tilde{\f}_p)^2 = +\infty$. First note that on $\mathbb{P}^2\setminus \{p\}$
$$\vol(\tilde{\omega})\pi_{\star} \mu= (\omega'+\omega_{FS}+dd^c \tilde{\f}_p)^2=2\omega'\wedge (\omega_{FS}+dd^c \tilde{\f}_p)+(\omega_{FS}+dd^c \tilde{\f}_p)^2. $$  Thus
\begin{eqnarray*}
\int_X (-\pistar \tilde{\f}_\varepsilon) d\mu &=& \int_{\mathbb{P}^2} (-\tilde{\f}_\varepsilon) d\pi_{\star}\mu \\
&=& \frac{1}{3} \left[2\int_{\mathbb{P}^2} (-\tilde{\f}_\varepsilon)\,\omega'\wedge (\omega_{FS}+dd^c \tilde{\f}_p)+\int_{\mathbb{P}^2} (-\tilde{\f}_\varepsilon) (\omega_{FS}+dd^c \tilde{\f}_p)^2 \right].
\end{eqnarray*}
We infer that $$\int_{\mathbb{B}(\frac{1}{2})\setminus \{(0,0)\}} |(-\log\|z\|)^{\varepsilon}| \,dd^c\log\|z\|\wedge dd^c [\chi(\log\|z\|)]=+\infty$$
where $\chi(t)=-(-t)^p$, hence the conclusion. Indeed on $\mathbb{B}(\frac{1}{2})\setminus\{(0,0)\}$,
$$dd^c\log\|z\|\wedge dd^c [\chi(\log\|z\|)]= \frac{A}{\|z\|^4} \chi''(\log\|z\|) dz_1\wedge d\bar{z}_1\wedge dz_2\wedge d\bar{z}_2$$ where $A$ is positive constant. Therefore we have
\begin{eqnarray*}
&&\int_{\mathbb{B}(\frac{1}{2})\setminus\{(0, 0)\}} \frac{1}{\|z\|^4 |\log\|z\||^{2-p-\varepsilon}} dz_1 \wedge d\bar{z}_1 \wedge dz_2 \wedge d\bar{z}_2\\
&&= C' \int_{0}^{\frac{1}{2}} \frac{1}{\rho (-\log\rho)^{2-p-\varepsilon}}\,\,d\rho \\
&&=C'\int_{-\log\frac{1}{2}}^{+\infty} \frac{1}{s^{2-p-\varepsilon}} \,ds=+\infty 
\end{eqnarray*}
since $2-p-\varepsilon\leq 1$.\\
A similar computation shows that $\pistar \tilde{\f}_p\in \cE(X,\tilde{\omega})$ and so $\mu\in \MA(\cE^1(X,\a))$ by construction.
\end{proof}


\end{document}